\documentclass[12pt, reqno]{amsart}
\usepackage{graphicx}
\usepackage{hyperref}
\usepackage[caption = false]{subfig}
\usepackage{amsthm}
\usepackage{amssymb}
\usepackage{mathrsfs}
\usepackage{mathtools}
\usepackage{enumerate}
\usepackage{amssymb}
\usepackage{wrapfig}
\usepackage{float}
\allowdisplaybreaks

\usepackage[twoside,paperwidth=190mm, paperheight=297mm, top=35mm, bottom=35mm, left=30mm, right=26mm, marginparsep=3mm, marginparwidth=40mm]{geometry}

\numberwithin{equation}{section}

\theoremstyle{plain}

\newtheorem{theorem}{Theorem}[section]
\newtheorem{corollary}[theorem]{Corollary}
\newtheorem{example}[theorem]{Example}
\newtheorem{lemma}{Lemma}[section]

\theoremstyle{definition}
\newtheorem{definition}[theorem]{Definition}
\theoremstyle{remark}
\newtheorem{remark}{Remark}[section]
\newtheorem{con}{Conjecture}[section]

\makeatother

\begin{document}

\title{ On certain Generalizations of $\mathcal{S}^*(\psi)$: II }
	\thanks{K. Gangania thanks to University Grant Commission, New-Delhi, India for providing Junior Research Fellowship under UGC-Ref. No.:1051/(CSIR-UGC NET JUNE 2017).}

	\author[Kamaljeet]{Kamaljeet Gangania}
	\address{Department of Applied Mathematics, Delhi Technological University,
		Delhi--110042, India}
	\email{gangania.m1991@gmail.com}
	
	\author[S. Sivaprasad Kumar]{S. Sivaprasad Kumar}
	\address{Department of Applied Mathematics, Delhi Technological University,
		Delhi--110042, India}
	\email{spkumar@dce.ac.in}

\maketitle	
	
\begin{abstract} 
	In this paper, we prove various radius results and obtain sufficient conditions using the convolution for the Ma-Minda classes $\mathcal{S}^*(\psi)$ and $\mathcal{C}(\psi)$ of starlike and convex analytic functions. We also obtain the Bohr radius for the class $ S_{f}(\psi):= \{g(z)=\sum_{k=1}^{\infty}b_k z^k : g \prec f  \}$ of subordinants, where $f\in \mathcal{S}^*(\psi).$ The results are improvements and generalizations of several well known results.
\end{abstract}
\vspace{0.5cm}

\maketitle
	
	\section{Introduction}
Let $\mathcal{A}$ denote the class of analytic functions of the form $f(z)=z+\sum_{k=2}^{\infty}a_kz^k$ in the open unit disk $\mathbb{D}:=\{z\in\mathbb{C}: |z|<1\}$. Using subordination,	Ma-Minda \cite{minda94} introduced the unified class of starlike and convex functions defined as follows:
\begin{equation}\label{mindaclass}
\mathcal{S}^*(\psi):= \biggl\{f\in \mathcal{A} : \frac{zf'(z)}{f(z)} \prec \psi(z) \biggl\} 
\end{equation}
and
\begin{equation*}
\mathcal{C}(\psi):= \biggl\{f\in \mathcal{A} : 1+\frac{zf''(z)}{f'(z)} \prec \psi(z) \biggl\},
\end{equation*}
where  $\psi$ is a Ma-Minda function, which is   analytic and univalent with $\Re{\psi(z)}>0$, $\psi'(0)>0$, $\psi(0)=1$ and $\psi(\mathbb{D})$ is symmetric about real axis. Note that $\psi \in \mathcal{P}$, the class of normalized Carath\'{e}odory functions. The class $\mathcal{S}^*(\psi)$  unifies various subclasses of starlike functions, which are obtained for an appropriate choice of $\psi$. See \cite{kumar-2019,Goel,Kumar-cardioid,mendi2exp,raina-2015,sokolconj-2009}. Ma-Minda discussed many properties of the class $\mathcal{S}^*(\psi)$, in particular, they proved the distortion theorem \cite[Theorem~2, p.162]{minda94} with some restriction on $\psi$, namely
\begin{equation}\label{minda-cond}
\underset{|z|=r}{\min}|\psi(z)|=\psi(-r)\quad \text{and}\quad \underset{|z|=r}{\max}|\psi(z)|=\psi(r).
\end{equation}
Importance of the above conditions can be seen in achieving the sharpness for majorization results in \cite{ganga_cmft2021}. In section~\ref{sec-3}, we modify the distortion theorem by relaxing this restriction on $\psi$ to obtain a more general result.
In 1914, Harald Bohr \cite{bohr1914} proved the following remarkable result related to the power series:
\begin{theorem}[\cite{bohr1914}]
	Let $g(z)=\sum_{k=0}^{\infty}a_kz^k$ be an analytic function in $\mathbb{D}$ and $|g(z)|<1$ for all $z\in \mathbb{D}$, then
	$\sum_{k=0}^{\infty}|a_k|r^k\leq1$
	for all $z\in \mathbb{D}$ with $|z|=r\leq1/3$.
\end{theorem}

Bohr actually proved the above result for $r\leq1/6$. Further Wiener, Riesz and Shur independently sharpened the result for $r\leq1/3$. Presently the Bohr inequality for functions mapping unit disk onto different domains, other than unit disk is an active area of research. For the recent development on Bohr-phenomenon, see the articles \cite{jain2019,bhowmik2018,boas1997,muhanna14,singh2002} and references therein.
The concept of Bohr-phenomenon in terms of subordination can be described as:
Let $f(z)=\sum_{k=0}^{\infty}a_kz^k$ and $g(z)=\sum_{k=0}^{\infty}b_kz^k$ are analytic in $\mathbb{D}$ and $f(\mathbb{D})=\Omega$. For a fixed $f$, consider a class of analytic functions $S(f):=\{g : g\prec f\}$ or equivalently $S(\Omega):=\{g : g(z)\in \Omega\}$. Then the class $S(f)$ is said to satisfy Bohr-phenomenon, if there exists a constant $r_0\in (0,1]$ satisfying the inequality $
\sum_{k=1}^{\infty}|b_k|r^k \leq d(f(0),\partial\Omega)$
for all $|z|=r\leq r_0$ and $g(z) \in S(f)$, where $d(f(0),\partial\Omega)$ denotes the Euclidean distance between $f(0)$ and the boundary of $\Omega=f(\mathbb{D})$. The largest such $r_0$ for which the inequality holds, is called the Bohr-radius. In  2014, Muhanna et al. \cite{muhanna14} proved  the Bohr-phenomenon for $ S(W_{\alpha})$, where  $W_{\alpha}:=\{w\in\mathbb{C} : |\arg{w}|<\alpha\pi/2, 1\leq \alpha \leq2\},$  which  is a Concave-wedge domain (or exterior of a compact convex set) and the class $R(\alpha,\beta, h)$ defined by $R(\alpha,\beta,h) := \{g\in \mathcal{A} : g(z)+\alpha z g'(z)+\beta z^2 g''(z)\prec h(z)\},$ where $h$ is a convex function (or starlike) and $R(\alpha,\beta, h)\subset S(h)$.	In 2018, Bhowmik and Das \cite{bhowmik2018} proved the Bohr-phenomenon for the classes given by
$S(f)=\{g\in\mathcal{A} : g\prec f\; \text{and}\; f\in \mu(\lambda) \}$, where  $\mu(\lambda)=\{f\in\mathcal{A} : |(z/f(z))^2f'(z)-1|<\lambda, 0<\lambda\leq1\}$ and $S(f)=\{g\in \mathcal{A}: g\prec f \;\text{and}\; f\in \mathcal{S}^*(\alpha), 0\leq\alpha\leq1/2\}$,
where $\mathcal{S}^*(\alpha)$ is the well-known class of starlike functions of order $\alpha$.

In Section~\ref{sec-4}, for any fixed  $f\in \mathcal{S}^*(\psi)$,  we introduce and study the Bohr-phenomenon inside the disk $|z|\leq1/3$ for the following class of analytic subordinants:
\begin{definition} 
	Let  $f\in \mathcal{S}^{*}(\psi)$. Then the class of analytic subordinants functions is defined as
	\begin{equation}\label{bohrclass}
	S_{f}(\psi):= \biggl\{g(z)=\sum_{k=1}^{\infty}b_k z^k : g \prec f \biggl \}.
	\end{equation}
\end{definition}
Note that $\mathcal{S}^*(\psi)\subset \bigcup_{f\in \mathcal{S}^*(\psi)} S_{f}(\psi)$. As an application, we obtain the Bohr-radius for the class $S(f)$, where $f\in \mathcal{S}^*({(1+Dz)}/{(1+Ez)})$, the class of Janowski starlike functions, with some additional restriction on $D$ and $E$ apart from $-1\leq E< D\leq1$. 

Now recall that the convolution of two power series $f(z)=z+\sum_{n=2}^{\infty}a_nz^n$ and $g(z)=z+\sum_{n=2}^{\infty}b_n z^n$ is defined by 
$$(f*g)(z)=z+\sum_{n=2}^{\infty}a_n b_n z^n.$$
Ruscheweyh and Sheil-Small~\cite{rus-small-1973} proved that if $f\in \mathcal{C}$ and $g\in \mathcal{C}$ ($g\in \mathcal{S}^*$), then $f*g \in \mathcal{C}$ ($f*g \in \mathcal{S}^*$). Later on Ma-Minda~\cite{minda94} proved that $f\in \mathcal{C}$ and $g\in \mathcal{S}^*(\psi)$, then $f*g \in \mathcal{S}^*(\psi)$, when $\psi(\mathbb{D})$ is convex. Szeg\"{o}~\cite{szeg-1928} in 1928 dicussed the radii of convexity for the sections $f_k(z)=z+\sum_{n=2}^{k}a_n z^n$ of the functions $f\in \mathcal{C}$, while in 1988 Silverman~\cite{Sil-1988} considered the radii of starlikeness of $f_k$. Also see Silverman et al.~\cite{SST-1978,SS-1985} work on convolution for the Janowski classes.  We shall consider these problems for the Ma-Minda classes $\mathcal{C}(\psi)$ and $\mathcal{S}^*(\psi)$ following the idea of Goodman and Schoenberg~\cite{Good-schoen1984}.
In section~\ref{sec-2}, we obtain the condition for functions to be in $\mathcal{S}^*(\psi)$ which is an extention of the Bulboac\u{a} and Tuneski~\cite{Bulboca-2003}. We also obtain several radius results and, necessary and sufficient conditions through the convolution for the classes $\mathcal{S}^{*}(\psi)$ and $\mathcal{C}(\psi)$ .

\section{\bf{Convolution properties and sufficient conditions for starlike and convex functions }}\label{sec-2}

In the following, we extend the results of the Bulboac\u{a} and Tuneski~\cite{Bulboca-2003} for the class $\mathcal{S}^{*}(\psi)$:
\begin{theorem}\label{bulextn}
	Let $h$ be analytic with $h(0)=0$, $h'(0)\neq0$. Suppose that $h$ satisfies 
	\begin{equation}\label{bul-cond}
	\Re\left(1+\frac{zh''(z)}{h(z)}\right) > -\frac{1}{2}
	\end{equation}
	and
	\begin{equation}\label{p}
	\frac{1}{z}\int_{0}^{z}h(t)dt \prec \frac{\psi(z)-1}{\psi(z)}.
	\end{equation}
	If  $f\in \mathcal{A}$, then 
	\begin{equation*}
	\frac{f(z)f''(z)}{(f'(z))^2} \prec h(z) \quad\text{implies} \quad f\in \mathcal{S}^{*}(\psi).
	\end{equation*}
\end{theorem}
\begin{proof}
	Using the result \cite[Theorem~3.1, p.~3]{Bulboca-2003}, we see that ${f(z)f''(z)}/{(f'(z))^2} \prec h(z)$ implies 
	\begin{equation*}
	\frac{1}{z}\int_{0}^{z}\left(1-\left(\frac{f(t)}{f'(t)}\right)'\right)dt = 1-\frac{f(z)}{zf'(z)} \prec \frac{1}{z}\int_{0}^{z}h(t)dt.
	\end{equation*}	
	From the above subordination, we have  
	\begin{equation*}
	\frac{f(z)}{zf'(z)} \prec	1-\frac{1}{z}\int_{0}^{z}h(t)dt.
	\end{equation*}
	Now to prove that $f\in \mathcal{S}^{*}(\psi)$, it suffices to consider
	\begin{equation*}
	1-\frac{1}{z}\int_{0}^{z}h(t)dt \prec \frac{1}{\psi(z)},
	\end{equation*} 
	which is equivalent to \eqref{p}. This completes the proof. \qed
\end{proof}

\begin{corollary}\label{Janpower}
	Let 	$-1\leq E< D\leq 1$ and $0<\beta \leq 1$ such that
	$$h(z)=1-\left( \frac{1+Ez}{1+Dz}  \right)^{\beta} \left(\frac{1+(D+E-\beta (D-E))z +DE z^2}{(1+Dz)(1+Ez)} \right)$$
	and satisfies \eqref{bul-cond}.
	If $f\in \mathcal{A}$ and 
	\begin{equation*}
	\frac{f(z)f''(z)}{(f'(z))^2} \prec h(z).
	\end{equation*}
	Then $f \in \mathcal{S}^{*} \left(\left( \frac{1+Dz}{1+Ez}  \right)^{\beta} \right).$
\end{corollary}

\begin{remark}
	\begin{enumerate}[$(i)$]
		\item Choosing $E=-1$ and $D=1$ in Corollary~\ref{Janpower} gives \cite[Corollary~3.5]{Bulboca-2003}. 
		\item  Choosing $E=0$ and $\beta=1$, then Corollary~\ref{Janpower} reduces to \cite[Example~4.5]{Bulboca-2003}. 
	\end{enumerate}
\end{remark}

Choosing $D=1-2\alpha$, $E=-1$ and $\beta=1$ in Corollary~\ref{Janpower} yields \cite[Corollary~3.4]{Bulboca-2003} (Note that $\alpha\neq 1$). Also, in \cite[Example~4.4]{Bulboca-2003} correct range of $\alpha$ is $[0,1)$. 
\begin{corollary}
	If $f\in \mathcal{A}$, $0\leq \alpha <1$ and 
	\begin{equation*}
	\frac{f(z)f''(z)}{(f'(z))^2} \prec \frac{2(1-\alpha)((1-2\alpha)z^2 +2z)}{(1+(1-2\alpha)z)^2}.
	\end{equation*}
	Then $f $ is starlike of order $\alpha$.
\end{corollary}

Considering $\psi(z)=\sqrt{1+c z}$, where  $0<c\leq 1$ in Theorem~\ref{bulextn}, we get the following:
\begin{corollary}
	Let $0< c< c_0<1$, where $c_0 \approx 0.845276$ is the unique positive root of the equation
	$$30-\frac{75}{2}c^2-\frac{201}{32}c^4=0.$$
	If $f\in \mathcal{A}$ and 
	\begin{equation*}
	\frac{f(z)f''(z)}{(f'(z))^2} \prec 1-\frac{2+ cz}{2(1+c z)^{3/2}}.
	\end{equation*}
	Then $f\in \mathcal{S}^{*}(\sqrt{1+c z})$.
\end{corollary}

Let us consider the following function in Theorem~\ref{bulextn}:
$$h(z)=1-\frac{e^{\lambda z} (1-\lambda z)}{e^{ 2\lambda z}},$$ 
where $0< \lambda\leq 1$ . Then 
$$1+\frac{zh''(z)}{h'(z)}=\frac{2-4\lambda z +\lambda^2 z^2}{2-\lambda z} \quad \text{and}\quad \Re\left(1+\frac{zh''(z)}{h'(z)}\right)=\frac{2-4\lambda +\lambda^2}{2-\lambda}.$$
\begin{corollary}
	Let $0< \lambda\leq \frac{1}{4}(9-\sqrt{33})<1$.
	If $f\in \mathcal{A}$ and 
	\begin{equation*}
	\frac{f(z)f''(z)}{(f'(z))^2} \prec 1-\frac{e^{\lambda z} (1-\lambda z)}{e^{ 2\lambda z}}.
	\end{equation*}
	Then $f\in \mathcal{S}^{*}(e^{\lambda z})$.
\end{corollary}

In 1985, Silverman and Silvia~\cite{SS-1985} obtained some necessary and sufficient conditions in terms of convolution operators for the functions to be in the Janowski classes, and generalized the results of Silverman et. al.~\cite{SST-1978}. We now extend the results for classes $\mathcal{S}^*(\psi)$ and $\mathcal{C}(\psi)$ as follows:
\begin{theorem}\label{conthm1}
	$f\in \mathcal{S}^*(\psi)$ if and only if 
	\begin{equation}\label{cnvn0}
	\frac{1}{z} \left( f(z)* \frac{z-\lambda z^2}{(1-z)^2}  \right) \neq0,
	\end{equation}
	where $\lambda=\psi(e^{i t})/(1-\psi(e^{i t}))$ and $t\in [0,2\pi)$.
\end{theorem}
\begin{proof}
	$f\in \mathcal{S}^*(\psi)$  if and only if $zf'(z)/f(z) \prec \psi(z)$, which is further equivalent to 
	\begin{equation*}
	\frac{zf'(z)}{f(z)} \neq \psi(e^{i t}), \quad (z\in \mathbb{D}, t\in [0,2\pi)),
	\end{equation*}
	( means that the values $zf'(z)/f(z)$ does not lie on the boundary $\Omega$ of the domain $\psi(\mathbb{D})$) which, using the fact that $f(z)/z \neq0$, can be equivalently written as 
	\begin{equation}\label{cnvn1}
	\frac{1}{z} \left( zf'(z)- \psi(e^{i t}) f(z)  \right) \neq0.
	\end{equation}
	Since
	\begin{equation*}
	zf'(z)= f(z)* \frac{z}{(1-z)^2} \quad \text{and} \quad f(z)=f(z)*\frac{z}{1-z}.
	\end{equation*}
	Therefore, \eqref{cnvn1} becomes
	\begin{equation*}
	\frac{1}{z} \left( f(z)* \frac{z-\psi(e^{i t})(z-z^2)}{(1-z)^2}  \right) \neq0,
	\end{equation*} 
	and further with a little computation, it reduces to \eqref{cnvn0}. \qed
\end{proof}

Note that we can also write   
$$\frac{1}{z} \left( f(z)* \frac{z-\lambda z^2}{(1-z)^2}  \right)=\frac{1}{z}(zf'(z)-\lambda(zf'(z)-f(z))).$$
Thus using the power series expansion of $f(z)$, \eqref{cnvn0} becomes
\begin{equation*}
\sum_{k=2}^{\infty}(\lambda(k-1)-k)a_k z^{k-1}\neq1,
\end{equation*}
where $\lambda$ is as defined in Theorem~\ref{conthm1}, and yields the following sufficient condition in terms of coefficients for the functions from $\mathcal{A}$ to be in $\mathcal{S}^*(\psi)$:
\begin{corollary}
	If $f\in \mathcal{A}$ satisfies $\sum_{k=2}^{\infty}|\lambda(k-1)-k||a_k|<1$. Then $f\in \mathcal{S}^*(\psi)$.
\end{corollary}

In particular, we have the following sufficient condition for the class of Janowski starlike functions:
\begin{corollary}
	The function  $f\in \mathcal{S}^*((1+Dz)/(1+Ez))$, if $$\sum_{k=2}^{\infty}\left( \frac{1+|D|}{D-E}+\left( \frac{1+|D|-E+D}{D-E} \right)k \right)|a_k|<1.$$
\end{corollary}

The following result with some rearrangment and specific values of $D$ and $E$ reduces to \cite[Theorem~6, Theorem~7]{SS-1985}.
\begin{corollary}
	Let $\psi(z)=(1+Dz)/(1+Ez)$, where $-1\leq E<D\leq1$. Then $f\in \mathcal{S}^*(\psi)$ if and only if 
	\begin{equation*}
	\frac{1}{z} \left( f(z)* \frac{z+\frac{\zeta+D}{D-E} z^2}{(1-z)^2}  \right) \neq0, \quad |\zeta|=1.
	\end{equation*}
\end{corollary}

\begin{corollary}
	Let $\psi(z)=1+ze^z$. Then $f\in \mathcal{S}^*(\psi)$ if and only if 
	\begin{equation*}
	\frac{1}{z} \left( f(z)* \frac{z+(1+\zeta e^{\zeta}) z^2}{(1-z)^2}  \right) \neq0, \quad |\zeta|=1.
	\end{equation*}
\end{corollary}

The class $\mathcal{S}^*(1+\sin{z})$ was introduced in \cite{kumar-2019}. 
\begin{corollary}
	Let $\psi(z)=1+\sin{z}$. Then $f\in \mathcal{S}^*(\psi)$ if and only if 
	\begin{equation*}
	\frac{1}{z} \left( f(z)* \frac{z-\arcsin{\zeta}(1+\sin{\zeta}) z^2}{(1-z)^2}  \right) \neq0, \quad |\zeta|=1.
	\end{equation*}
\end{corollary}

Note that $f\in \mathcal{S}^*(\psi)$ if and only if $g(z)=\int_{0}^{z}\frac{f(t)}{t}dt \in \mathcal{C}(\psi)$. Therefore, the condition given in \eqref{cnvn0} is equivalent to the following:
\begin{equation}\label{cnvn2}
\frac{1}{z} \left( zg'(z) * \frac{z-\lambda z^2}{(1-z)^2}  \right).
\end{equation}
Now using the convolution fact that $zg'(z)*f(z)= g(z)* zf'(z)$ in \eqref{cnvn2}, we obtain the following result, which is the convex analogue of Theorem~\ref{conthm1}:
\begin{theorem}\label{conthm2}
	$f\in \mathcal{C}(\psi)$ if and only if 
	\begin{equation*}
	\frac{1}{z} \left( f(z)* \frac{z+(1-2\lambda) z^2}{(1-z)^3}  \right) \neq0,
	\end{equation*}
	where $\lambda$ is as defined in Theorem~\ref{conthm1}.
\end{theorem}

Note that we can also write   
$$\frac{1}{z} \left( f(z)* \frac{z+(1-2\lambda) z^2}{(1-z)^3}  \right)  =f'(z)+(1-\lambda)zf''(z).$$
Thus, using the power series expansion of $f(z)$, we equivalently get
\begin{equation*}
\sum_{k=2}^{\infty}(\lambda(k-1)-k)ka_k z^{k-1}\neq1,
\end{equation*}
which gives the following sufficient condition in terms of coefficient's:

\begin{corollary}
	If $f\in \mathcal{A}$ satisfies $\sum_{k=2}^{\infty}k|\lambda(k-1)-k||a_k|<1$. Then $f\in \mathcal{C}(\psi)$.
\end{corollary}

In particular, we have the following sufficient condition for the class of Janowski convex functions:
\begin{corollary}
	The function  $f\in \mathcal{C}((1+Dz)/(1+Ez))$, if $$\sum_{k=2}^{\infty}\left( \frac{1+|D|}{D-E}+\left( \frac{1+|D|-E+D}{D-E} \right)k \right)k|a_k|<1.$$
\end{corollary}

\begin{corollary}
	Let $\psi(z)=(1+Dz)/(1+Ez)$, where $-1\leq E<D\leq1$. Then $f\in \mathcal{C}(\psi)$ if and only if 
	\begin{equation*}
	\frac{1}{z} \left( f(z)* \frac{z+\frac{(3D-E)+2\zeta}{D-E} z^2}{(1-z)^3}  \right) \neq0, \quad |\zeta|=1.
	\end{equation*}
\end{corollary}
\begin{corollary}
	Let $\psi(z)=1+ze^z$. Then $f\in \mathcal{C}(\psi)$ if and only if 
	\begin{equation*}
	\frac{1}{z} \left( f(z)* \frac{z+(3+2\zeta e^{\zeta}) z^2}{(1-z)^3}  \right) \neq0, \quad |\zeta|=1.
	\end{equation*}
\end{corollary}
\begin{corollary}
	Let $\psi(z)=1+\sin{z}$. Then $f\in \mathcal{C}(\psi)$ if and only if 
	\begin{equation*}
	\frac{1}{z} \left( f(z)* \frac{z-\arcsin{\zeta}(2+\sin{\zeta}) z^2}{(1-z)^3}  \right) \neq0, \quad |\zeta|=1.
	\end{equation*}
\end{corollary}

\begin{remark}
	Let us consider the class introduced in \cite{ganga-iranian}:
	\begin{equation*}
	\mathcal{F}(q):=  \left\{ f\in \mathcal{A}: \frac{zf'(z)}{f(z)}\prec q(z) \right \},
	\end{equation*}
	where $q(z)=1+\sum_{n=2}^{\infty}q_n z^n$ is univalent. Further, assume that $f(z)/z\neq0$.  Then Theorem~\ref{conthm1} also holds for the class $\mathcal{F}(q)$, and Theorem~\ref{conthm2} for its convex analogue.
	
\end{remark}

It is well known that the class $\mathcal{C}(\psi)$ is closed under convolution. Also the class $\mathcal{S}^{*}(\psi)$ is closed under convolution when convoluted with convex functions, when $\psi(\mathbb{D})$ is convex. With this motivation, the following result provides the largest radius $r_0$ such that in $|z|<r_0$ the class $\mathcal{S}^{*}(\psi)$ is closed under convolution.
\begin{theorem}\label{conthm3}
	Let $r_{\psi}$ be the largest radius such that $F(z)=z+\sum_{n=2}^{\infty}n^2 z^n$ belongs to $\mathcal{S}^*(\psi)$ for $|z|<r_{\psi}$. If $f, g \in \mathcal{S}^*(\psi)$, where $\psi$ is convex. Then $f*g$ belongs to $\mathcal{S}^*(\psi)$ for $|z|< r_{\psi}$. The radius is best possible.
\end{theorem}
\begin{proof}
	Let $f(z)= z+\sum_{n=2}^{\infty}a_n z^n$ and $g(z)=z+\sum_{n=2}^{\infty}b_n z^n$. Then
	\begin{align*}
	G(z) &:= f(z)*g(z)\\
	&= \left(z+ \sum_{n=2}^{\infty}n^2 z^n\right) * \left(z+ \sum_{n=2}^{\infty}\frac{a_n}{n} z^n\right) * \left(z+ \sum_{n=2}^{\infty} \frac{b_n}{n} z^n\right)\\
	&= F(z)*  \left(z+ \sum_{n=2}^{\infty}\frac{a_n}{n} z^n\right) * \left(z+ \sum_{n=2}^{\infty} \frac{b_n}{n} z^n\right).
	\end{align*}
	Recall that 
	\begin{equation*}
	f\in \mathcal{S}^*(\psi) \quad \text{if and only if} \quad \int_{0}^{z}\frac{f(t)}{t}dt \in \mathcal{C}(\psi).
	\end{equation*}
	Similary for the function $g$. Therefore, $ z+ \sum_{n=2}^{\infty}\frac{a_n}{n} z^n$ and $ z+ \sum_{n=2}^{\infty}\frac{b_n}{n} z^n$ belong to $\mathcal{C}(\psi)$. Now let
	\begin{align*}
	H(z) &:= \left(z+ \sum_{n=2}^{\infty}\frac{a_n}{n} z^n\right) * \left(z+ \sum_{n=2}^{\infty} \frac{b_n}{n} z^n\right)\\
	&= z+\sum_{n=2}^{\infty} \frac{a_n b_n}{n^2}z^n \in  \mathcal{C}(\psi) \subseteq  \mathcal{C}
	\end{align*}
	so that 
	\begin{equation*}
	G(z)= F(z)* H(z).
	\end{equation*}
	Now from the hypothesis, we have
	\begin{equation}\label{f*g1}
	\frac{F(r_{\psi} z)}{r_{\psi}} \in  \mathcal{S}^*(\psi),
	\end{equation}
	that is, $F$ belongs to $ \mathcal{S}^*(\psi)$  for $|z|< r_{\psi}$. Since $H \in \mathcal{C}$. Therefore, using \eqref{f*g1}, we get
	\begin{equation}\label{B(z)}
	B(z):= \left( \frac{F(r_{\psi} z)}{r_{\psi}} \right)* H(z) \in \mathcal{S}^*(\psi),
	\end{equation}
	whenever $\psi$ is convex. Thus, we conclude that
	\begin{equation*}
	G(z)=r_{\psi} B\left( \dfrac{z}{r_{\psi}} \right) \in \mathcal{S}^*(\psi), \quad (|z|<r_{\psi}).
	\end{equation*}
	Note that the radius $r_\psi$ is independent of the choices of the functions $f$ and $g$, and thus the sharpness of the result follows from \eqref{f*g1}. \qed
\end{proof}

\begin{remark}
	Note that 
	\begin{equation*}
	F(z)= \frac{z(1+z)}{(1-z)^3} =z+\sum_{n=2}^{\infty}n^2 z^n.
	\end{equation*}
	Now the logarithmic differentiation of the above yields:
	\begin{equation}\label{F1}
	\frac{zF'(z)}{F(z)}= \frac{1+4z+z^2}{1-z^2}= \frac{1+z}{1-z}-\frac{1}{1+z}+\frac{1}{1-z}.
	\end{equation}
	It follows from \eqref{F1} that
	\begin{equation}\label{Fk1}
	\Re \left( \frac{zF'(z)}{F(z)}  \right) \geq \frac{1-4r+r^2}{1-r^2}.
	\end{equation}
	Moreover, the following sharp inequality also holds:
	\begin{equation}\label{Fk2}
	\left|\frac{zF'(z)}{F(z)}-1   \right| \leq \frac{2r(2+r)}{1-r^2}.
	\end{equation}
	Furthermore, the following inequality also holds:
	\begin{equation}\label{Fk3}
	\left| \frac{zF'(z)}{F(z)}- \frac{1+r^2}{1-r^2}  \right| \leq \frac{4r}{1-r^2}.
	\end{equation}
\end{remark}

Now applying Theorem~\ref{conthm3}, we get the following result: 
\begin{corollary}\label{cor-conthm3}
	Let $f, g \in \mathcal{S}^*(\psi)$, where $\psi$ is convex. Then $f*g$ belongs to $\mathcal{S}^*(\psi)$ for $|z|< r_{\psi}$, where
	\begin{enumerate} [$(i)$]
		\item $r_\psi= (2-\sqrt{3+\alpha^2})/(1+\alpha)$, when $\psi(z)= (1+(1-2\alpha)z)/(1-z)$.
		\item $r_\psi=(-2+\sqrt{5})/(1+\sqrt{2})$, when $\psi(z)= \sqrt{1+z}$.
		\item $r_\psi=(-2+\sqrt{4+b(2+b)})/(2+b)$, where $b=(e-1)/(e+1)$, when $\psi(z)=2/(1+e^{-z})$.
		\item $r_\psi=(2-\sqrt{4-b^2})/b$, where $b=\sin{\pi\gamma/2}$, when $\psi(z)=((1+z)/(1-z))^{\gamma}$.
	\end{enumerate}
	The radii are sharp. 
\end{corollary}
\begin{proof}
	The part $(i)$ follows using the inequality \eqref{Fk1} such that $(1-4r+r^2)(1-r^2)\geq \alpha$, which holds for $r\leq (2-\sqrt{3+\alpha^2})/(1+\alpha)$. Further
	\begin{equation*}
	\frac{z_0F'(z_0)}{F(z_0)}=\alpha \quad \text{for}\quad z_0= \frac{\sqrt{3+\alpha^2}-2}{1+\alpha},
	\end{equation*}
	implies the sharpness of the radius. 
	Note that the disks $\{w: |w-1|< \sqrt{2}-1  \}$ and $\{w: |w-1|< (e-1)/(e+1)  \}$ are contained in $\phi(\mathbb{D})$ where $\phi(z)=\sqrt{1+z}$ and $2/(1+e^{-z})$, respectively. Therefore, the parts $(ii)$ and $(iii)$ follow using the inequality \eqref{Fk2} such that
	\begin{equation*}
	\frac{2r(2+r)}{1-r^2} \leq \sqrt{2}-1 \quad \text{and} \quad  \frac{2r(2+r)}{1-r^2} \leq \frac{e-1}{e+1},
	\end{equation*}
	which holds for $r\leq (-2+\sqrt{5})/(1+\sqrt{2})$ and $r\leq (-2+\sqrt{4+b(2+b)})/(2+b)$, where $b=(e-1)/(e+1)$ respectively. Since
	\begin{align*}
	\frac{z_0F'(z_0)}{F(z_0)}= 1+\sqrt{2} \quad \text{for} \quad z_0=\frac{-2+\sqrt{5}}{1+\sqrt{2}}
	\end{align*}
	and
	\begin{align*}
	\frac{z_0F'(z_0)}{F(z_0)}=\frac{2}{1+e} \quad \text{for} \quad z_0=(\sqrt{4+b(2+b)}-2)/(2+b),
	\end{align*}
	therefore the radii obtained are sharp. Part $(iv)$ follows by using the inequality \eqref{Fk3} and the fact that the disk $\{w: |w-a|<r_a  \}$ is contained in the sector $|\arg{w}|\leq \pi\gamma/2$, whenever $r_a\leq a\sin(\pi\gamma/2)$. \qed
\end{proof}	

The following lemma was introduced by Kumar and Gangania~\cite{KK-IJPM} to obtain certain radius constants (see \cite[p.12-14]{KK-IJPM}) related to the operators like Livingston and Bernardi etc. to cover the case when $\psi$ is starlike but not convex. 
\begin{lemma}
	Let $r_c$ be the radius of convexity of $\psi$. If $g\in \mathcal{C}$ and $f\in \mathcal{S}^*(\psi)$. Then $f*g\in \mathcal{S}^*(\psi)$ for $|z|<r_\psi= \min \{r_c, 1 \}$.
\end{lemma}

Now using the above lemma, we may write \eqref{B(z)} as follows:
\begin{equation*}
B(z):= \left( \frac{F(r_0 z)}{r_0} \right)* H(z) \in \mathcal{S}^*(\psi),
\end{equation*}
where $r_0= \min \{r_\psi, r_c  \}$. Note that $r_c=(3-\sqrt{5})/2$ when $\psi(z)=1+ze^z$. Thus, we have the following result:
\begin{corollary}
	Let $f, g \in \mathcal{S}^*(1+ze^z)$. Then $f*g$ belongs to $\mathcal{S}^*(1+ze^z)$ for $|z|< (2e-\sqrt{4e^2-2e+1})/(2e-1)\approx 0.0957$.
	The radius is sharp.
\end{corollary}
\begin{corollary}
	Let $f, g \in \mathcal{S}^*(1+\sin{z})$. Then $f*g$ belongs to $\mathcal{S}^*(1+\sin{z})$ for $|z|< (\sqrt{4+\sin{1}(2+\sin{1})}-2)/(2+\sin{1})\approx 0.1858$.
	The radius is sharp.
\end{corollary}

In the next theorem, we continue to extend the ideas of Szeg\"{o}~\cite{szeg-1928} and Silverman~\cite{Sil-1988} on the starlikeness and convexity of sections $f_k(z)=z+\sum_{n=2}^{k}a_n z^n$ of $f$ in $\mathcal{S}^*(\psi)$ and $\mathcal{C}(\psi)$, respectively.
\begin{theorem}\label{sections}
	Let $g_k(z)=(z-z^{k+1})/(1-z)$. If $f\in \mathcal{C}(\psi)$, where $\psi$ be convex. Then
	\begin{enumerate}[(i)]
		\item  $f_k \in \mathcal{C}(\psi)$ in $|z|<r_{0}$, where $r_{0}$ is the radius of convexity of $g_k$.
		
		\item  $f_k \in \mathcal{S}^*(\psi)$ in $|z|<r_{\psi}$, whenever $g_k$ belongs to $\mathcal{S}^*(\psi)$ in $|z|<r_{\psi}$. 
	\end{enumerate}
	The radii are best possible.
\end{theorem}
\begin{proof}
	Let $t=r_0$ in the first part and $r_{\psi}$ in second part, respectively.Then the proof follows by observing that $f_k(z)=t h_k(z/t)$, where $h_k(z)=f(z)*\dfrac{g_k(t z)}{t}$.  \qed
\end{proof}

\begin{remark}
	If we choose $\psi(z)=(1+z)/(1-z)$. Then Theorem~\ref{sections}-(ii) reduces to the Silverman's result \cite[Theorem~1, p.~1192]{Sil-1988}.
\end{remark}		

Jackson~\cite{jackson-1908} introduced and studied the $q$-derivative defined as
\begin{equation*}
d_q f(z):=\frac{f(qz)-f(z)}{(q-1)z}=\frac{1}{z}\left(z+\sum_{n=2}^{\infty}[n]_q a_n z^n \right), \quad z\neq0
\end{equation*}
and $d_q f(0)=f'(0)$, where $[n]_q= \frac{1-q^n}{1-q}$. 

\begin{theorem}
	Let $r_\psi$ be the largest radius in $(0,1]$ and $q\in (0,1)$ such that 
	\begin{equation*}
	\frac{z}{(1-qz)(1-z)} \in \mathcal{S}^*(\psi) \quad \text{for} \quad |z|<r_\psi,
	\end{equation*}
	where $\psi$ is convex. If $f\in \mathcal{C}$, then we have
	\begin{equation*}
	zd_{q} f(z)= z+\sum_{n=2}^{\infty} [n]_q a_n z^n \in \mathcal{S}^*(\psi) \quad \text{for} \quad |z|<r_\psi.
	\end{equation*}
	The radius is best possible.
\end{theorem}
\begin{proof}
	Observe that for each $q\in \mathbb{C}$ where $|q|\leq1$, $q\neq1$, we have
	\begin{align*}
	h_q(z) &=\frac{1}{1-q} \log \left( \frac{1-qz}{1-z}  \right) =\sum_{n=1}^{\infty} \left( \frac{1-q^n}{1-q}  \right) \frac{z^n}{n} \\
	&=\sum_{n=1}^{\infty} \frac{[n]_q}{n} {z^n} \in \mathcal{C},
	\end{align*}
	which implies that 
	\begin{equation*}
	zh'_q(z)= \frac{z}{(1-qz)(1-z)}= \sum_{n=1}^{\infty}[n]_q z^n \in \mathcal{S}^*.
	\end{equation*}
	We note that
	\begin{align*}
	zd_q f(z) &= \left( z+\sum_{n=2}^{\infty}a_n z^n  \right)* \left( \sum_{n=1}^{\infty} [n]_q z^n   \right)\\
	&= f(z)* \frac{z}{(1-qz)(1-z)}\\
	& =f(z)* zh'_q(z).
	\end{align*}
	Now for simplicity, let us write $H(z)=zh'_q(z)$. Then from hypothesis we see that
	\begin{equation}\label{dq}
	\frac{H(r_\psi z)}{r_\psi} \in \mathcal{S}^*(\psi).
	\end{equation} 
	Since $f\in \mathcal{C}$, and $\mathcal{C}* \mathcal{S}^*(\psi)=  \mathcal{S}^*(\psi)$ whenever $\psi$ is convex. Therefore, we have
	\begin{equation*}
	f(z)* \frac{H(r_\psi z)}{r_\psi} \in \mathcal{S}^*(\psi),
	\end{equation*}
	which is equivalent to  saying that
	\begin{equation*}
	zd_{q} f(z) \in \mathcal{S}^*(\psi) \quad \text{for} \quad |z|<r_\psi.
	\end{equation*}
	From the proof, we note that the sharpness of the radius $r_\psi$ follows from \eqref{dq}. \qed
\end{proof}

\begin{remark}
	From the function $H(z)={z}/((1-qz)(1-z))$, we have:
	\begin{equation}\label{Hq1}
	\frac{zH'(z)}{H(z)}= 1+\frac{z}{1-z}+\frac{qz}{1-qz}.
	\end{equation}
	It follows from \eqref{Hq1} that
	\begin{equation}\label{Janw-Hq1}
	\Re \left( \frac{zH'(z)}{H(z)}  \right) \geq \frac{1-qr^2}{(1+r)(1+qr)}.
	\end{equation}
	Moreover, the following sharp inequality also holds:
	\begin{equation}\label{Lem-Sig-Hq1}
	\left|\frac{zH'(z)}{H(z)}-1   \right| \leq \frac{r(1+q-2qr)}{(1-r)(1-qr)}.
	\end{equation}
\end{remark}

Now proceeding in a similar way as in Corollary~\ref{cor-conthm3} using \eqref{Janw-Hq1} when $\psi(z)= (1+(1-2\alpha)z)/(1-z)$, and \eqref{Lem-Sig-Hq1} when $\psi(z)=\sqrt{1+z}$ and $2/(1+e^{-z})$, we have
\begin{corollary}
	If $f\in \mathcal{C}$, then for all $0<q<1$, we have
	\begin{equation*}
	zd_{q} f(z)= z+\sum_{n=2}^{\infty} [n]_q a_n z^n \in \mathcal{S}^*(\psi) \quad \text{for} \quad |z|<r_\psi,
	\end{equation*}	
	where
	\begin{enumerate} [$(i)$]
		\item $r_\psi=(\sqrt{\alpha^2(1-q^2)+4q}-\alpha(q+1))/(2q(1+\alpha))$, when $\psi(z)= (1+(1-2\alpha)z)/(1-z)$ for $\alpha\geq (1-q)/2(1+q)$.
		\item $r_\psi=((1+q)-\sqrt{1+q^2})/(q\sqrt{2}(\sqrt{2}+1))$, when $\psi(z)= \sqrt{1+z}$.
		\item $r_\psi=((1+q)(1+b)-\sqrt{((1+q)(1+b))^2-4bq(2+b)})/(2q(2+b))$, where $b=(e-1)/(e+1)$, when $\psi(z)=2/(1+e^{-z})$.
	\end{enumerate}
	The radii are sharp. 
\end{corollary}

\begin{remark}
	Note that part~(i) of the above corollary includes the result [\cite{PS-2020}, Theorem~2.1]. Also $r_\psi=1$ for $\alpha\in[0,(1-q)/2(1+q)]$.
\end{remark}

\section{\bf{Distortion theorem }}\label{sec-3}
Ma-Minda \cite{minda94} proved the distortion theorem for the class $\mathcal{S}^*(\psi)$ with some restriction on $\psi$, namely $|\psi(z)|$ attains its maximum and minimum value  respectively at $z=r$ and $z=-r$, see eq.~\eqref{minda-cond}. Now what  if $\psi$ does not satisfy the condition \eqref{minda-cond} and why the condition \eqref{minda-cond} is so important? To answer this, we first need to recall the following result:
\begin{lemma}\emph{(\cite{minda94})}\label{grth}
	Let $f\in \mathcal{S}^*(\psi)$ and $|z_0|=r<1$. Then $-f_0(-r)\leq|f(z_0)|\leq f_0(r).$
	Equality holds for some $z_0\neq0$ if and only if $f$ is a rotation of $f_0$, where $zf_0(z)/f_0(z)=\psi(z)$ such that
	\begin{equation}\label{int-rep}
	f_0(z)=z\exp{\int_{0}^{z}\frac{\psi(t)-1}{t}dt}.
	\end{equation}
	
\end{lemma}

We see that a  Ma-Minda starlike function, in general, need not satisfy the condition~\eqref{minda-cond}.  To examine the same, let us consider two different  Ma-Minda starlike functions, namely
$\psi_1(z):=z+\sqrt{1+z^2}$ and $\psi_2(z):=1+ze^z.$
The unit disk images of $\psi_1$ and $\psi_2$ are displayed in figure~\ref{fg1} and figure~\ref{fg2}.
\begin{figure}[h]
	
	{\includegraphics[scale=0.17]{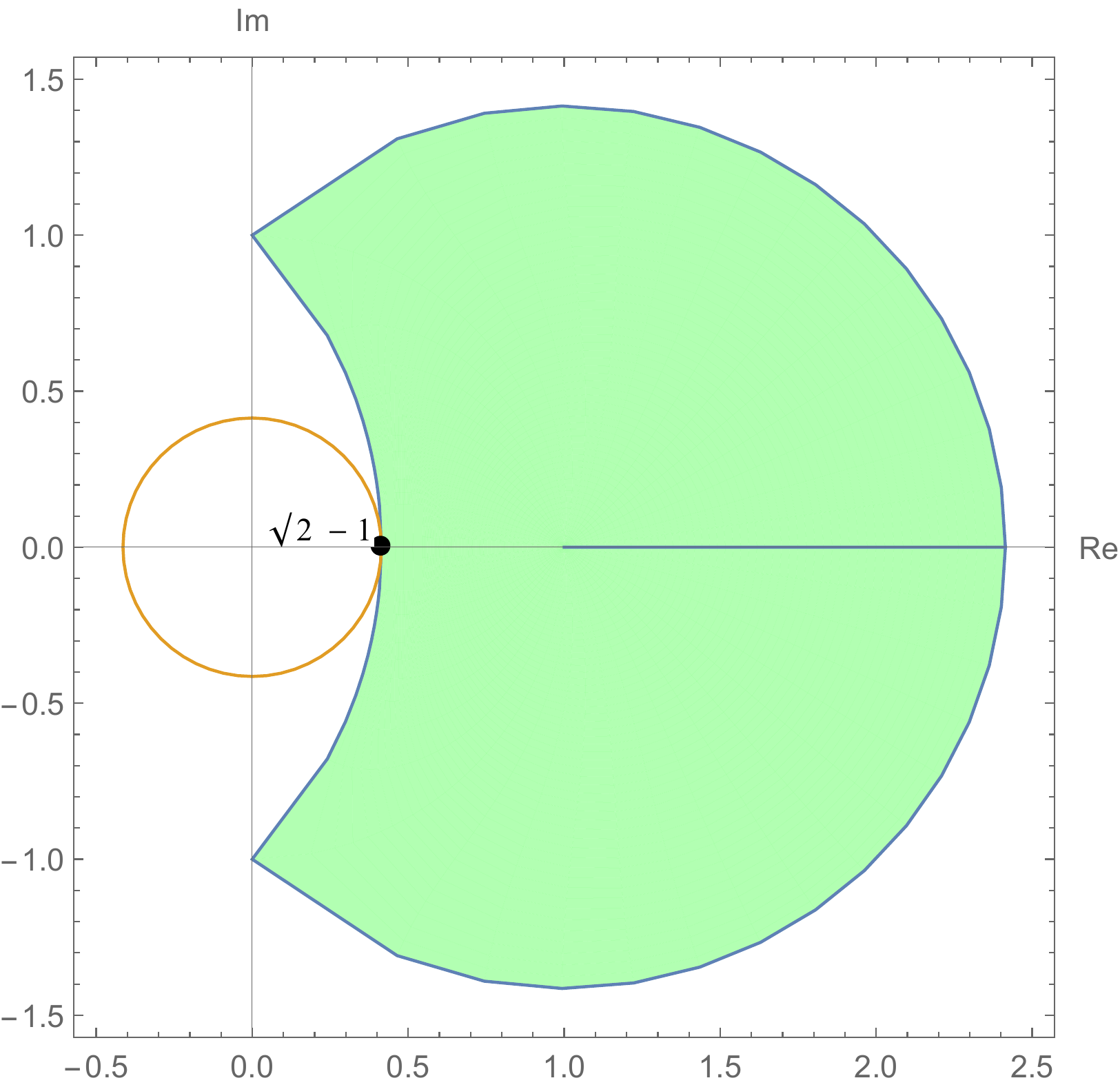}}{\caption{$\psi_1(z):=z+\sqrt{1+z^2}$}\label{fg1}}
	
	Let $\underset{|z|=r}{\min}|\psi_2(z)|=:\gamma_i(r),$ where
	$z=r_i e^{i\theta},\; 0 \leq \theta\leq\pi$,
	then from table~\ref{table}, we have
	$\gamma_1(1)= 0.372412,  $
	$\gamma_2(4/5)= 0.527912, $
	$\gamma_3(2/3)= 0.611553,$
	$\gamma_4(1/2)= 0.693287,$
	$\gamma_5(r)= 1-re^{-r},$
	where $r\leq(3-\sqrt{5})/2.$\\~\\
	
		{\includegraphics[scale=0.155]{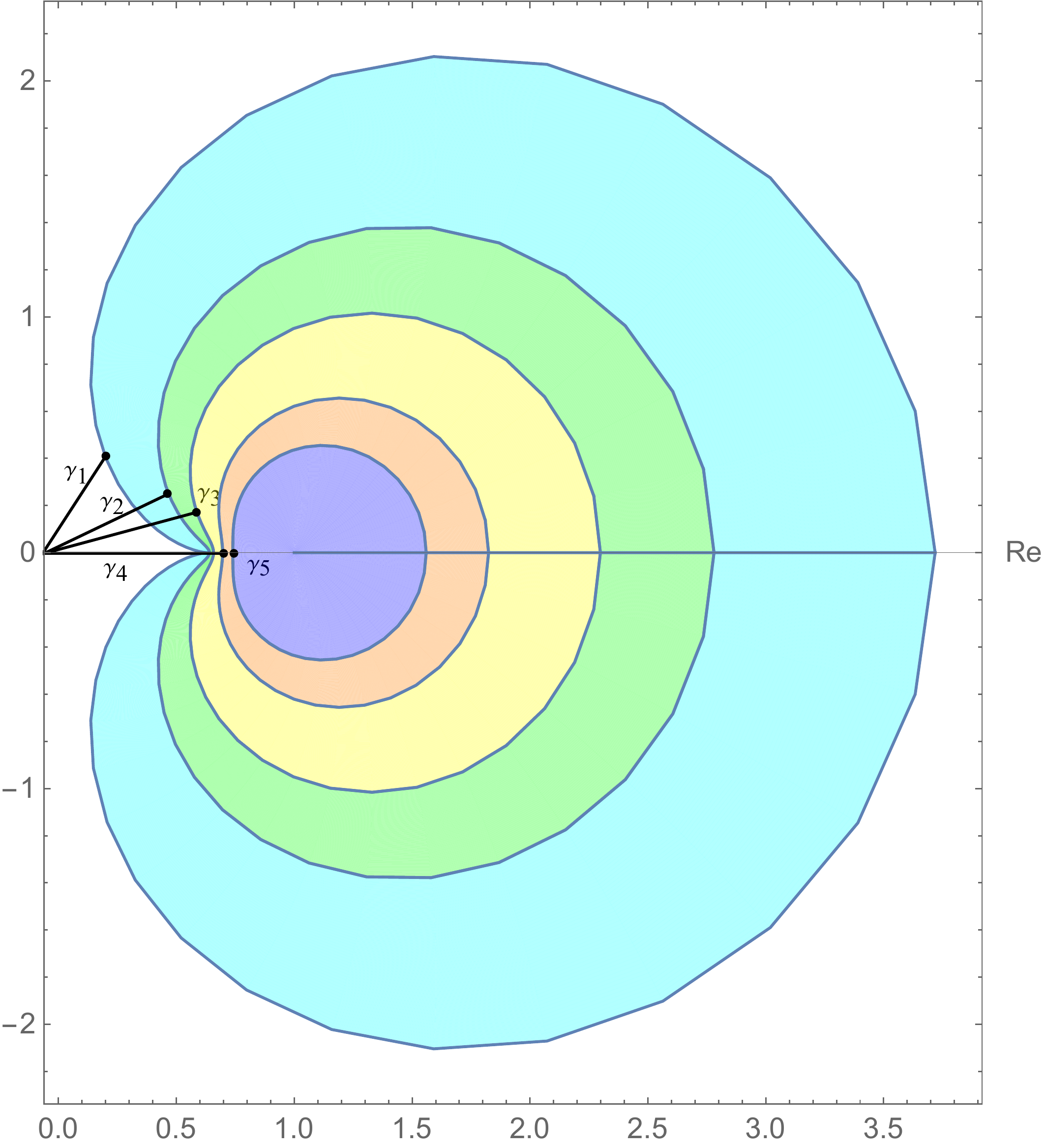}}{\caption{$\psi_2(z):=1+ze^z$}\label{fg2}}
		{\includegraphics[scale=0.45]{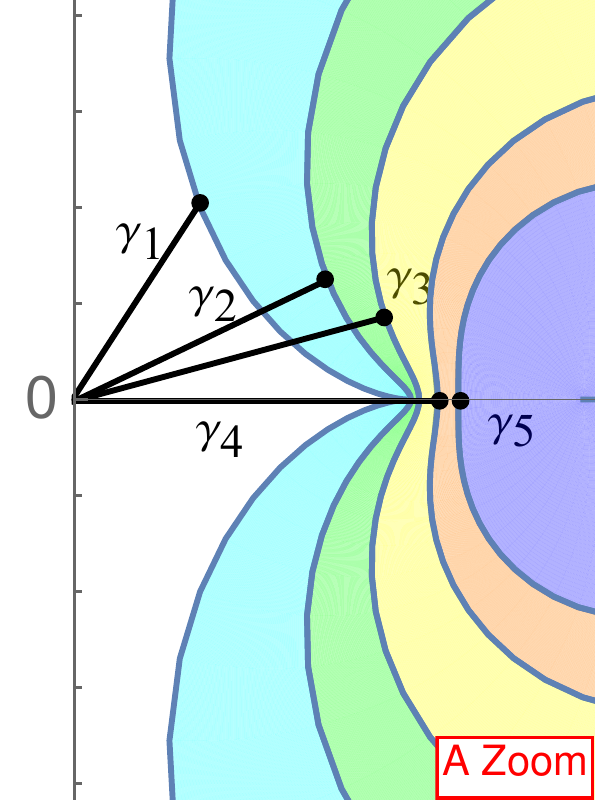}}{\caption{A zoom of figure~\ref{fg2}\label{fg3}}}

\end{figure}

We know that the radius of a circle, centered at origin and touching only the boundary points of a image domain of a complex function, yields the optimal values of the modulus of the function. For example, see figure~\ref{fg1} to locate the lower bound of the modulus for a crescent function. Therefore it  is evident from figure~\ref{fg1}   that both the bounds $\psi_1(-r)$ and $\psi_1(r)$ of $|\psi_1|$ are attained on the real line and we have $\psi_1(-r)\leq|\psi_1(z)|\leq \psi_1(r)$ for each $|z|=r$. Whereas, from figure~\ref{fg2}, we see that only the upper bound $\psi_2(r)$ of $|\psi_2|$  is attained  on the real line and $|\psi_2(z)|\leq \psi_2(r)$ for each $|z|=r$. Although both $\psi_1$ and $\psi_2$ are Ma-Minda functions but the distortion theorem of Ma-Minda \cite[theorem~2, p.~162]{minda94} does not accommodate the function $\psi_2$, as the lower bound for $|\psi_2(z)|$ is not attained on the real line for all $|z|=r> (3-\sqrt{5})/2$, see figure~\ref{fg3}. To overcome this limitation, we modify the distortion theorem, wherein we theoretically assume the modulus bounds of the function and obtain a more general result. Thus the Ma-Minda functions, for which modulus bounds are not attained on the real line but could be computed  can now be entertained  for distortion theorem using the following result:

\begin{theorem}[Modified Distortion Theorem]\label{distthm}
	Let $\psi$ be a Ma-Minda function. Assume that $\underset{|z|=r}{\min}|\psi(z)|=|\psi(z_1)|$ and $\underset{|z|=r}{\max}|\psi(z)|=|\psi(z_2)|$, where $z_1=re^{i\theta_1}$ and $z_2=re^{i\theta_2}$ for some $\theta_1, \theta_2\in [0,\pi]$.
	Let $f\in \mathcal{S}^*(\psi)$ and $|z_0|=r<1$. Then
	\begin{equation}\label{disteq}
	|\psi(z_1)|\left(\frac{-f_0(-r)}{r}\right)\leq|f'(z_0)|\leq \left(\frac{f_0(r)}{r}\right)|\psi(z_2)|.
	\end{equation}
\end{theorem}
\begin{proof}
	Let $p(z)=zf'(z)/f(z)$. Then $f\in \mathcal{S}^*(\psi)$	 if and only if $p(z)\prec \psi(z)$. Using a result \cite[Theorem~1, p.161]{minda94}, we have
	\begin{equation}\label{dist-equation}
	\frac{f(z)}{z}\prec \frac{f_0(z)}{z},
	\end{equation}
	where $f_0$ is given by \eqref{int-rep}.
	Now using Maximum-Minimum principle of modulus, \eqref{dist-equation} and by Lemma \ref{grth}, $-f_0(-r)/r\leq|f(z_0)/z|\leq f_0(r)/r$, we easily obtain for $|z_0|=r$
	\begin{align*}
	|\psi(z_1)|\left(\frac{-f_0(-r)}{r}\right)&=\min_{|z|=r}|\psi(z)|\min_{|z|=r}\left|\frac{f_0(z)}{z}\right|\\
	&\leq\left|p(z_0)\frac{f(z_0)}{z_0}\right|=|f'(z_0)|\\
	&\leq\max_{|z|=r}|\psi(z)|\max_{|z|=r}\left|\frac{f_0(z)}{z}\right|= \left(\frac{f_0(r)}{r}\right)|\psi(z_2)|,
	\end{align*}
	that is,
	\begin{equation*}
	|\psi(z_1)|\left(\frac{-f_0(-r)}{r}\right)\leq|f'(z_0)|\leq\left(\frac{f_0(r)}{r}\right)|\psi(z_2)|,
	\end{equation*}
	where $z_1$ and $z_2$ are as defined in the hypothesis. Hence the result.\qed
\end{proof}

To illustrate Theorem~~\ref{distthm}, we consider the function $\psi(z)=1+ze^z$. Then we have the following expression for its modulus:
\begin{equation}\label{cardioid-dist}
|\psi(z)|=\sqrt{1+ r e^{r\cos\theta}(r e^{r\cos\theta} + 2\cos(\theta+r\sin\theta))}.
\end{equation}
Using equation~\eqref{cardioid-dist} and Theorem~\ref{distthm}, we obtain the following table, providing the minimum for various choices of $r$:
\begin{table}[h]
	\caption{The lower bounds for $ |1+ze^z| $ for different choices of $r=|z|$. }\label{table}
	\centering	
	\begin{tabular}{|c|c|c|c|}
		\hline
			 $r$ & $0\leq \theta_1 \leq \pi$ & $ |\psi(re^{i\theta_1 })|$ & $m(r,\theta_1 )=|\psi(re^{i\theta_1 })|(-f_0(-r)/r)$    \\
		\hline
			1 &  1.88438 & 0.372412  &0.197923 \\
		\hline
		 4/5 &  2.01859 & 0.527912 & 0.304374\\
		\hline
		2/3 & 2.17677 & 0.611553 & 0.375966\\
		\hline
		1/2 & 2.58169 & 0.693287 & 0.467769 \\
		\hline
		$r\leq (3-\sqrt{5})/2$ & $\pi$ & $\psi_2(-r)$ &$f'_0(-r)$ \\
		\hline
	\end{tabular}
\end{table}	

Now using Theorem~\ref{distthm}, we obtain the following distortion theorem for the class $\mathcal{S}^{*}(1+ze^z)$:
\begin{example}
	Let $\psi(z)=1+ze^z$ and $\underset{|z|=r}{\min}|\psi(z)|=|\psi(z_1)|$, where $z_1=re^{i\theta_1}$ for some $\theta_1\in [0,\pi]$.
	Let $f\in \mathcal{S}^*(\psi)$ and $|z_0|=r<1$. Then
	\begin{equation*}
	m(r,\theta_1)\leq|f'(z_0)|\leq f'_0(r), \quad \left(r>\tfrac{3-\sqrt{5}}{2}\right)
	\end{equation*}
	and
	\begin{equation*}
	f'_0(-r)\leq|f'(z_0)|\leq f'_0(r), \quad \left(r\leq\tfrac{3-\sqrt{5}}{2}\right),
	\end{equation*}
	where $f_0(z)= z \exp(e^z-1)$ and $m(r,\theta_1)$ is provided in table~\ref{table} for some specific values of $r$.
\end{example}

\begin{remark}
	In Theorem \ref{distthm}, if we assume that $\theta_1=\pi$ and $\theta_2=0$, then extremes in equation \eqref{disteq} simplifies to $f'_0(-r)$ and $f'_0(r)$, respectively since $zf'_0(z)/f_0(z)=\psi(z)$. Thus, the extremes in the equation \eqref{disteq} are in terms of $r$ alone and also lead to the sharp bounds.
	Consequently, we obtain the following distortion theorem of Ma-Minda \cite{minda94} as a special case of  Theorem \ref{distthm} :\\
	
	{\it	Let $\min_{|z|=r}|\psi(z)|=\psi(-r)$ and $\max_{|z|=r}|\psi(z)|=\psi(r)$. If $f\in \mathcal{S}^*(\psi)$ and $|z_0|=r<1$. Then
		\begin{equation*}
		f'_0(-r)=\psi(-r)\frac{f_0(-r)}{-r}\leq|f'(z_0)|\leq \frac{f_0(r)}{r}\psi(r)=f'_0(r).
		\end{equation*}
		Equality holds for some $z_0\neq0$ if and only if $f$ is a rotation of $f_0$.}	
\end{remark}

\section{Bohr-Phenomenon for functions in $S_{f}(\psi)$}\label{sec-4}

Note that {``the Bohr radius of the class $\mathcal{X}$ is at least $r_{x}$'',} this holds for every result in this section. In general, Bohr radius is estimated for a specific class provided the sharp coefficients bounds of the functions in that class are known. For instance, consider the class of starlike univalent functions, where we have the sharp coefficient bounds: $|a_n|\leq n$. However, for most of the Ma-Minda subclasses, the better coefficients bounds are yet not known. Hence, we encounter the following problem, especially in context of Ma-minda classes, which we deal here to a certain extent:\\
{\bf Problem: {\it If coefficients bounds are not known, how one can find a good lower estimate for the Bohr radius of a given class?}}

To readily understand the above problem, consider the class $\mathcal{S}^{*}(1+ze^z)$, where the sharp coefficients bounds for functions in this class are unknown. In this situation, how one can find the Bohr radius for this class or is there any way out with the lower bounds all alone? Here below we state Theorem~\ref{bohr}, where we find a solution for this problem. Note that the Bohr radius $3-2\sqrt{2}\approx 0.1713$ for the class $\mathcal{S}^{*}$ serves as a lower bound for the class $S_{f}(\psi)$ and is also a special case of Theorem~\ref{bohr}. 

Let $\mathbb{B}(0,r):=\{z\in\mathbb{C}: |z|<r  \},$  $g(z)=\sum_{k=1}^{\infty}b_k z^k,$   $\mathcal{S}^*(\psi)$
and   $S_{f}(\psi)$ as defined in  \eqref{mindaclass} and \eqref{bohrclass} respectively.
For any $g\in S_{f}(\psi),$ we find the radius $r_b$ so that $S_{f}(\psi)$ obey the following Bohr-phenomenon:
\begin{equation}\label{bohrphenomenon}
\sum_{k=1}^{\infty}|b_k|r^k \leq d(f(0),\partial\Omega)\quad \text{for}\quad|z|=r\leq r_b,
\end{equation}
where  $d(f(0),\partial\Omega)$ denotes the Euclidean distance between $f(0)$ and the boundary of $\Omega=f(\mathbb{D})$. Now we prove our main result:

\begin{theorem}\label{bohr}
	Let $r_{*}$ be the Koebe-radius for the class $\mathcal{S}^*(\psi),$  $f_0(z)$ be given by the equation~\eqref{int-rep} and $g(z)=\sum_{k=1}^{\infty}b_k z^k \in S_{f}(\psi)$. Assume  $f_0(z)=z+\sum_{n=2}^{\infty}t_n z^n$ and $\hat{f}_0(r)=r+\sum_{n=2}^{\infty}|t_n|r^n$.
	Then   $S_{f}(\psi)$ satisfies the Bohr-phenomenon
	\begin{equation}\label{compare}
	\sum_{k=1}^{\infty}|b_k|r^k \leq d(f(0),\partial\Omega),\quad \text{for}\; |z|=r\leq r_b,
	\end{equation}
	where $r_b=\min\{r_0, 1/3 \}$, $\Omega=f(\mathbb{D})$ and $r_0$ is the least positive root of the equation
	$$\hat{f}_0(r)=r_{*}.$$
	The result is sharp when $r_b=r_0$ and $t_n>0$.
\end{theorem}
\begin{proof}
	Since  $g\in S_{f}(\psi)$, we have $g\prec f$ for a fixed $f\in \mathcal{S}^*(\psi)$ . By letting $r$ tends to $1$ in Lemma~\ref{grth}, we obtain the Koebe-radius $r_{*}=-f_0(-1)$. Therefore
	$\mathbb{B}(0,r_{*}) \subset f(\mathbb{D})$, which implies that 
	$$r_{*}\leq d(0,\partial\Omega)=|f(z)| \quad \text{for} \quad |z|=1.$$
	Also using the result \cite[Theorem~1, p.161]{minda94}, we have
	\begin{equation}\label{f-f0}
	\frac{f(z)}{z}\prec \frac{f_0(z)}{z}.
	\end{equation}
	Recall the result \cite[Lemma~1, p.1090]{bhowmik2018}, which reads as:
	let $f$ and $g$ be analytic in $\mathbb{D}$ with $g\prec f,$ where
	$f(z)=\sum_{n=0}^{\infty}a_n z^n$ and $ g(z)= \sum_{k=0}^{\infty}b_k z^k.$
	Then
	$\sum_{k=0}^{\infty}|b_k|r^k \leq \sum_{n=0}^{\infty}|a_n|r^n$ for $ |z|=r\leq1/3.$
	Now using the result for  $g\prec f$ and  \eqref{f-f0}, we have
	\begin{equation*}
	\sum_{k=1}^{\infty}|b_k|r^k \leq	r+\sum_{n=2}^{\infty}|a_n|r^n \leq \hat{f}_0(r)\quad\text{for}\; |z|=r\leq1/3.
	\end{equation*}
	Finally, to establish the inequality \eqref{compare}, it is enough to show $\hat{f}_0(r) \leq r_{*}.$
	But this holds whenever $r\leq r_0$, where $r_0$ is the least positive root of the equation $\hat{f}_0(r)=r_{*}.$
	The existence of the root $r_0$ is ensured by the relations 
	$$\hat{f}_0(1)\geq |f_0(1)|\geq r_{*} \quad \text {and} \quad \hat{f}_0(0)<r_{*}.$$
	Thus, if $r_b=\min\{r_0, 1/3 \}$ then $\sum_{k=1}^{\infty}|b_k|r^k \leq d(0, \partial{\Omega}) $ holds. The case of sharpness follows for the function $f_0$. \qed
\end{proof}

\begin{remark}
	Let us further assume that the coefficients $B_n$ of $\psi$ are positive. Then the function $f_0(z)=z+\sum_{n=2}^{\infty}t_n z^n$ defined by integral representation \eqref{int-rep} can be written as
	$$f_0(z)=z \exp\left(\sum_{n=1}^{\infty}\frac{B_n}{n}z^n\right),$$
	which implies
	$f_0(r)=\hat{f}_0(r)$ for $|z|=r.$
\end{remark}	

From the proof of Theorem~\ref{bohr}, we have the following:
\begin{theorem}
	Let $r_{*}$ be the Koebe-radius for the class $\mathcal{S}^*(\psi),$  $f_0(z)$ be given by the equation~\eqref{int-rep} and $f(z)=z+\sum_{n=2}^{\infty}a_n z^n \in \mathcal{S}^*(\psi)$. Assume  $f_0(z)=z+\sum_{n=2}^{\infty}t_n z^n$ and $\hat{f}_0(r)=r+\sum_{n=2}^{\infty}|t_n|r^n$.
	Then   $\mathcal{S}^*(\psi)$ satisfies the Bohr-phenomenon
	\begin{equation*}
	r+\sum_{n=2}^{\infty}|a_n|r^n \leq d(f(0),\partial\Omega),\quad \text{for}\; |z|=r\leq r_b,
	\end{equation*}
	where $r_b=\min\{r_0, 1/3 \}$, $\Omega=f(\mathbb{D})$ and $r_0$ is the least positive root of the equation
	$$\hat{f}_0(r)=r_{*}.$$
	The result is sharp when $r_b=r_0$ and $t_n>0$.
\end{theorem}

If we choose $\psi(z)=(1+Dz)/(1+Ez)$, $-1\leq E<D\leq1$, then $\mathcal{S}^*(\psi)$ denotes the class of  Janowski starlike functions and we have
\begin{equation}\label{jb}
r_{*}=
\left\{
\begin{array}
{lr}
(1-E)^{\frac{D-E}{E}}, &  E\neq0; \\
e^{-D},   & E=0.
\end{array}
\right.
\end{equation}
and
\begin{equation}\label{jf}
f_0(z)=
\left\{
\begin{array}
{lr}
z(1+Ez)^{\frac{D-E}{E}}, &  E\neq0; \\
z\exp(Dz),   & E=0.
\end{array}
\right.
\end{equation}
Observe that if $E\neq0$, the $n$-th $(n\geq2)$ coefficients of $f_0(z)$ is given by
\begin{equation}\label{t_n}
t_n=\prod_{k=2}^{\infty}\frac{D-(k-1)E}{(k-1)!}.
\end{equation}
Thus from Theorem~\ref{bohr}, we have the following result:
\begin{corollary}\label{janouwski}
	Let $\psi(z)=(1+Dz)/(1+Ez)$, $-1\leq E<D\leq1$. Then $S_{f}(\psi)$ (and $\mathcal{S}^{*}(\psi)$) satisfies the Bohr-phenomenon \eqref{bohrphenomenon} for $|z|=r\leq r_b,$
	where $r_b=\min\{r_0, 1/3 \}$ and $r_0$ is the least positive root of the equation
	$$r+\sum_{n=2}^{\infty}|t_n|r^n-(1-E)^{\tfrac{D-E}{E}}=0,$$
	where $t_n$ is as defined in \eqref{t_n}.
\end{corollary}

Note that for the Janowski class, sharp coefficients bounds in general are not known. Now as an application of Corollary \ref{janouwski}, we obtain the following result when $t_n>0$:
\begin{corollary} {\bf{(Bohr-radius with Janowski class)} }\label{jrslt}
	Let $\psi(z)=(1+Dz)/(1+Ez)$, $-1\leq E<D \leq1$.
	\begin{itemize}
		\item [$(i)$]
		If $E=0$ and $D\geq \frac{3}{4}\log{3} $. Then $S_{f}(\psi)$ (and $\mathcal{S}^{*}(\psi)$) satisfies the Bohr-phenomenon \eqref{bohrphenomenon}
		for $|z|=r\leq r_0,$
		where $r_0$ is the only real root of the equation
		\begin{equation}\label{eb0}
		1-re^{D(1+r)}=0.
		\end{equation}
		
		\item [$(ii)$] If $E\neq0$ and further satisfies
		\begin{equation}\label{DE}
		3(1-E)^{\frac{D-E}{E}} \leq (1+E/3)^{\frac{D-E}{E}}.
		\end{equation}
		Then $S_{f}(\psi)$ (and $\mathcal{S}^{*}(\psi)$) satisfies the Bohr-phenomenon
		\eqref{bohrphenomenon} for $|z|=r\leq r_0,
		$
		where $r_0$ is the only real root of the equation
		\begin{equation}\label{eb}
		(1-E)^{\frac{D-E}{E}}-r(1+Er)^{\frac{D-E}{E}}=0.
		\end{equation}
	\end{itemize}
	The result is sharp for the function $f_0$ defined in \eqref{jf}.
\end{corollary}
\begin{proof}
	{{(i):}} Since $E=0,$ we have $r_*=e^{-D}$. Moreover $\hat{f}_0(r)=f_0(r)= r\exp(Dr)$. Now we need to show
	\begin{equation}\label{B0S}
	r\exp(Dr)\leq e^{-D}
	\end{equation}
	or equivalently $T(r):=1-r e^{D(1+r)}\geq0$
	holds for $r\leq r_0$. Which obviously holds for $\frac{3}{4}\log{3} \leq D \leq1$. Since $d(f_0(0),\partial f_0(\mathbb{D}))= r_*$, therefore we see from inequality \eqref{B0S} that Bohr-radius is sharp for the function $f_0$ given by \eqref{jf}.\\
	
	{{(ii):}} Proceeding as in case~(i), it is sufficient to show the inequality
	\begin{equation}\label{ABS}
	r(1+Er)^{\frac{D-E}{E}}\leq (1-E)^{\frac{D-E}{E}}
	\end{equation}
	or equivalently $g(r):=(1-E)^{\frac{D-E}{E}}-r(1+Er)^{\frac{D-E}{E}} \geq0$
	holds for $r\leq r_0$. Which obviously follows whenever $D$ and $E$ satisfies \eqref{DE}. In view of the inequality~\eqref{ABS}, the sharp Bohr-radius is achieved for the function $f_0$ given by \eqref{jf}.  \qed
\end{proof}

\begin{remark}{{(Bohr-radius with starlike functions of order $\alpha$)} }
	Let $\psi(z):=(1+(1-2\alpha)z)/(1-z),$ where  $0\leq \alpha<1$. We see $\mathcal{S}^*(\psi) :=\mathcal{S}^*(\alpha)$ and for this class, we have
	$$r_{*}=\frac{1}{2^{2(1-\alpha)}} \quad \text{and} \quad f_0(z)=\frac{z}{(1-z)^{2(1-\alpha)}}.$$
	Observe that here $\hat{f}_0(r)=f_0(r)$. Now as an application of Corrollary \ref{jrslt}, we obtain the result due to Bhowmik et al. \cite{bhowmik2018}, namely:\\
	
	{\it If $0\leq\alpha\leq1/2$. Then $S_{f}(\psi)$ satisfies the Bohr-phenomenon $\sum_{k=1}^{\infty}|b_k|r^k \leq d(f(0),\partial\Omega),\; \text{for}\; |z|=r\leq r_b,$
		where $r_b=\min\{r_0, 1/3 \}=r_0$ and $r_0$ is the only real root of the equation
		$ {(1-r)^{2(1-\alpha)}}/r=2^{2(1-\alpha)}.$
		The result is sharp.}
\end{remark}

Now form the above remark, in particular, we have:
\begin{corollary}
	If $0\leq\alpha\leq1/2$. Then the class $\mathcal{S}^*((1+(1-2\alpha)z)/(1-z))$ satisfies the Bohr-phenomenon  \eqref{bohrphenomenon} for $|z|=r\leq r_0$, where $r_0$ is the only real root of the equation
	$$ {(1-r)^{2(1-\alpha)}}/r=2^{2(1-\alpha)}.$$
	The result is sharp. In particular, the Bohr radius for the class $\mathcal{S}^{*}$ is $3-2\sqrt{2}\approx 0.1713$.
\end{corollary}

If we choose $\psi(z)=\sqrt{1+z}$, then $\mathcal{S}^{*}(\psi):=\mathcal{SL}^{*}$, the class of lemniscate starlike functions and for this class we have:
\begin{equation}\label{lemiscate-conj}
f_0(z)=\frac{4z \exp(2\sqrt{1+z}-2)}{(1+\sqrt{1+z})^2} \quad \text{and}\quad  r_{*}=-f_0(-1)\approx 0.541341.
\end{equation}
Also in this case $\hat{f}_0(r)=f_0(r)$ and therefore, we obtain the following corollary:
\begin{corollary}
	The class $S_{f}(\psi)$ (and $\mathcal{SL}^{*}$), where $\psi(z)=\sqrt{1+z}$ satisfies the Bohr-phenomenon  \eqref{bohrphenomenon} for $|z|=r\leq 1/3.$
\end{corollary}

If we consider $\psi(z)=1+ze^z$, then $\mathcal{S}^{*}(\psi):=\mathcal{S}^{*}_{\wp}$, the class of cardioid starlike functions introduced in \cite{Kumar-cardioid} and for this class, we have:
\begin{equation}\label{cardioid-conj}
f_0(z)=z\exp(e^z-1) \quad \text{and}\quad r_{*}=-f_0(-1)\approx0.531464.
\end{equation}
Here we can also see that $\hat{f}_0(r)=f_0(r)$ and we obtain the following corollary:
\begin{corollary}\label{cardioid}
	The class $S_{f}(\psi)$ (and $\mathcal{S}^{*}_{\wp}$), where $\psi(z)=1+ze^z$ satisfies the Bohr-phenomenon \eqref{bohrphenomenon} for $|z|=r\leq 1/3.$
\end{corollary}

Ali et al.~\cite{jain2019} also showed that the coefficient bound of a function in  a class have a role in the estimation of the Bohr-radius. Observed that for each $f\in \mathcal{S}^{*}(\psi),$ the class $S_{f}(\psi)$ satisfies the Bohr-phenomenon for $r\leq \min(1/3, r_0)$, where $r_0$ is the least positive root of $\hat{f}_0(r)-r_{*}=0$. Since $\mathcal{S}^*(\psi)\subset \bigcup_{f\in \mathcal{S}^*(\psi)} S_{f}(\psi) $, therefore the Bohr-radius for the class $\mathcal{S}^{*}(\psi)$ is $r\geq \min(1/3, r_0).$ In Corollary~\ref{cardioid}, we find  $r_0\approx0.349681$ (an upper bound for Bohr radius), which is almost close to $1/3\approx0.33333$ and is the unique root of $f_0(r)-r_{*}=0$.
Moreover, the bound for the coefficients of the functions belonging to $\mathcal{S}^{*}_{\wp}$ and $\mathcal{SL}^{*}$ have been conjectured \cite{Kumar-cardioid,sokolconj-2009} with the extremals given in \eqref{cardioid-conj} and \eqref{lemiscate-conj} respectively. Thus by using Theorem~\ref{bohr} and the approach dealt in \cite{jain2019} (assuming that conjectures are true), we propose the following conjectures:

\begin{con}
	The Bohr-radius for the class $\mathcal{S}^{*}_{\wp}$ is $r_0\approx0.349681$ which is the unique root in $(0,1)$ of the equation $$re^{e^r}=e^{1/e}.$$
\end{con}
\begin{con}
	The Bohr-radius for the class $\mathcal{SL}^{*}$ is $r_0\approx0.439229$, which the unique root in $(0,1)$ of the equation
	$$e^2 r \exp(2\sqrt{1+r}-2)=(1+\sqrt{1+r})^2.$$ 
\end{con}


%
\section*{Conflict of interest}

The authors declare that they have no conflict of interest.


\begin{thebibliography}{99}
	%
	
	\bibitem{jain2019} R. M. Ali, N. K. Jain\ and\ V. Ravichandran, Bohr radius for classes of analytic functions, Results Math. {74}(4) (2019), Art. 179, 13 pp.
	
	\bibitem{bhowmik2018}  B. Bhowmik and N. Das, Bohr phenomenon for subordinating families of certain univalent
	functions, J. Math. Anal. Appl. 462(2) (2018), 1087–1098.
	
	\bibitem{boas1997} H. P. Boas\ and\ D. Khavinson, Bohr's power series theorem in several variables, Proc. Amer. Math. Soc. { 125}(10) (1997), 2975--2979.
	
	\bibitem{bohr1914}  H. Bohr, A theorem concerning power series, Proc. London Math. Soc. 13(2) (1914), 1–5.
	
	\bibitem{Bulboca-2003} T. Bulboac\u{a}\ and\ N. Tuneski, New criteria for starlikeness and strongly starlikeness, Mathematica {\bf 43(66)} (2001), no.~1, 11--22 (2003).
	
	\bibitem{kumar-2019} N. E. Cho, V. Kumar, S. S. Kumar,  V. Ravichandran, Radius problems for starlike functions associated with the sine function, Bull. Iranian Math. Soc. {45}(1) (2019), 213--232.
	
	\bibitem{ganga_cmft2021} K. Gangania\ and\ S. S. Kumar, On Certain Generalizations of $\mathcal{S}^*(\psi)$, Comput. Methods Funct. Theory (2021), https://doi.org/10.1007/s40315-021-00386-5.
	
	\bibitem{Goel}  P. Goel and  S. S. Kumar, Certain Class of Starlike Functions Associated with Modified Sigmoid Function, Bull. Malays. Math. Sci. Soc. (2019) https://doi.org/10.1007/s40840-019-00784-y.
	
	\bibitem{jackson-1908} F. H. Jackson, On  q -functions and certain difference operator. Trans. R. Soc. Edinb. 46, 253–281 (1908)
	
	\bibitem{Kumar-cardioid}  S. S. Kumar\ and\ G. Kamaljeet, A cardioid domain and starlike functions. Anal. Math. Phys. {11}(2) (2021), 13 pp.  https://doi.org/10.1007/s13324-021-00483-7.
	
	\bibitem{ganga-iranian} S. S. Kumar\ and\ G. Kamaljeet, On Geometrical Properties of Certain Analytic functions, Iran. J. Sci. Technol. Trans. A Sci. (2020), https://doi.org/10.1007/s40995-021-01116-1
	
	\bibitem{KK-IJPM} S. S. Kumar\ and\ G. Kamaljeet, Subordination and Radius Problems for Certain Starlike functions (communicated), arXiv:2007.07816v1.
	
	\bibitem{minda94} W. C. Ma and D. Minda: A unified treatment of some special classes of univalent functions, in {\it Proceedings of the Conference on Complex Analysis (Tianjin, 1992)}, 157--169, Conf. Proc. Lecture Notes Anal., I Int. Press, Cambridge, MA.
	
	\bibitem{mendi2exp} R. Mendiratta,  S. Nagpal,  V. Ravichandran, On a subclass of strongly starlike functions associated with exponential function, Bull. Malays. Math. Sci. Soc. {38}(1) (2015), 365--386.
	
	\bibitem{muhanna14} Y. A. Muhanna, R. M. Ali, Z. C. Ng\ and\  S. F. M. Hansi, Bohr radius for subordinating families of analytic functions and bounded harmonic mappings, J. Math. Anal. Appl. { 420}(1) (2014), 124--136.
	
	\bibitem{singh2002} 	V. I. Paulsen, G. Popescu\ and\ D. Singh, On Bohr's inequality, Proc. London Math. Soc.  { 85}(2) (2002), 493--512.
	
	\bibitem{PS-2020} K. Piejko\ and\ J. Sok\'{o}\l, On convolution and $q$-calculus, Bol. Soc. Mat. Mex. (3) {\bf 26} (2020), no.~2, 349--359.
	
	\bibitem{raina-2015} R. K. Raina and J. Sok\'{o}\l, Some properties related to a certain class of starlike functions, C. R. Math. Acad. Sci. Paris. {353}(11), 973--978(2015)
	
	\bibitem{rus-small-1973} St. Ruscheweyh\ and\ T. Sheil-Small, Hadamard products of Schlicht functions and the P\'{o}lya-Schoenberg conjecture, Comment. Math. Helv. {\bf 48} (1973), 119--135.
	
	
	\bibitem{SST-1978} H. Silverman, E. M. Silvia\ and\ D. Telage, Convolution conditions for convexity starlikeness and spiral-likeness, Math. Z. {\bf 162} (1978), no.~2, 125--130.
	
	\bibitem{SS-1985} H. Silverman\ and\ E. M. Silvia, Subclasses of starlike functions subordinate to convex functions, Canad. J. Math. {\bf 37} (1985), no.~1, 48--61.
	
	\bibitem{Sil-1988} H. Silverman, Radii problems for sections of convex functions, Proc. Amer. Math. Soc. {\bf 104} (1988), no.~4, 1191--1196.
	
	\bibitem{sokolconj-2009} J. Sok\'{o}\l, Coefficient estimates in a class of strongly starlike functions, Kyungpook Math. J. {49}(2) (2009), 349--353.
	
	\bibitem{szeg-1928} G. Szeg\"{o}, Zur Theorie der schlichten Abbildungen, Math. Ann. {\bf 100} (1928), no.~1, 188--211.
	\bibitem{Good-schoen1984}A. W. Goodman\ and\ I. J. Schoenberg, On a theorem of Szegő on univalent convex maps of the unit circle. J. Analyse Math. 44 (1984/85), 200--204.
	
	
	
\end{thebibliography}


\end{document}